\documentclass[12pt]{amsart}
\usepackage{amsmath, amsfonts, amssymb, amsthm,hyperref,mathtools,array}
\usepackage[T1]{fontenc}
\hypersetup{hypertex=true,
	colorlinks=true,
	linkcolor=blue,
	anchorcolor=blue,
	citecolor=blue}
\usepackage{bm}
\allowdisplaybreaks[4]
\def\lcm{\rm{lcm}}
\def\ord{{\rm ord}}

\def\a{{\bm a}}

\def\v{{\bm v}}
\def\w{{\bm w}}
\def\x{{\bm x}}

\def\0{{\bm 0}}

\def\1{{\bf 1}}

\def\bad{{\rm bad}}
\def\good{{\rm good}}

\def\lcm{{\rm lcm}}
\def\diam{{\rm diam}}
\def\pmod #1{\ ({\rm{mod}}\ #1)}
\def\mod #1{\ {\rm mod}\ #1}
\def\Ack{\medskip\noindent {\bf Acknowledgments}}

\theoremstyle{plain}
\newtheorem{theorem}{Theorem}[section]
\newtheorem{lemma}{Lemma}
\newtheorem{corollary}{Corollary}

\newtheorem{definition}{Definition}

\theoremstyle{remark}

\makeatletter
\@namedef{subjclassname@2020}{%
	\textup{2020} Mathematics Subject Classification}
\makeatother
\vspace{4mm}

\begin{document}
	
	\title[Arithmetic progressions in dense subsets of primitive elements of finite fields]
	{Arithmetic progressions in dense subsets of primitive elements of finite fields}
	\author[H.-L. Wu]{Hai-Liang Wu}
	
	\address {(Hai-Liang Wu) School of Science, Nanjing University of Posts and Telecommunications, Nanjing 210023, People's Republic of China}
	\email{\tt whl.math@smail.nju.edu.cn}

	\keywords{arithmetic progressions, Szemer\' edi's theorem, primitive elements, finite fields. 
		\newline \indent 2020 {\it Mathematics Subject Classification}. Primary 11L40, 11B25; Secondary 11A07.
		\newline \indent This research was supported by the National Natural Science Foundation of China (Grant No. 12671009).}
	
	\begin{abstract}
	Let $p$ be a prime and let $\mathcal{P}_p$ be the set of primitive elements of $\mathbb{F}_p$. Inspired by the work of Cohen, Oliveira e Silva and Trudgian on consecutive primitive elements, and by the work of Chang on arithmetic progressions in multiplicative subgroups of finite fields, in this paper, using estimates for multiplicative character sums over systems of linear forms together with the relative Szemer\'edi's theorem of Conlon, Fox and Zhao, we show that, for every fixed integer $k\ge3$ and any fixed number $0<\delta\le 1$, each subset $A\subseteq\mathcal{P}_p$ with $\#A\ge\delta\#\mathcal{P}_p$ contains a nontrivial $k$-term arithmetic progression, provided that $p$ is sufficiently large. Also, by applying Behrend's theorem, we construct a large subset $A'$ of $\mathcal{P}_p$ with $\#A'=\lfloor(\#\mathcal{P}_p)^{1-\varepsilon}\rfloor$ such that $A'$ contains no nontrivial $k$-term arithmetic progressions. 
	\end{abstract}
	\maketitle

	\section{Introduction}
	\setcounter{lemma}{0}
	\setcounter{theorem}{0}
	\setcounter{equation}{0}
	\setcounter{conjecture}{0}
	\setcounter{remark}{0}
	\setcounter{corollary}{0}
	
	\subsection{Notation} For any positive integer $N\ge2$, let $\mathbb{Z}/N\mathbb{Z}$ be the ring of residue classes modulo $N$. Throughout this paper, $p$ denotes a prime and 
	$$\mathbb{F}_p=\mathbb{Z}/p\mathbb{Z}=\{x\mod{p\mathbb{Z}}: 0\le x\le p-1\}$$ 
	is the finite field with $p$ elements. Let $\mathbb{F}_p^*=\mathbb{F}_p\setminus\{0\}$ be the multiplicative cyclic group of all nonzero elements of $\mathbb{F}_p$. An element $g\in\mathbb{F}_p^*$ is said to be a primitive element of $\mathbb{F}_p$ if $g$ generates the group $\mathbb{F}_p^*$. The set of all primitive elements of $\mathbb{F}_p$ is written as $\mathcal{P}_p$. The group of all multiplicative characters of $\mathbb{F}_p$ is denoted by $\widehat{\mathbb{F}_p^*}$, and we use $\chi_0$ to denote the trivial character. Given any multiplicative character $\chi: \mathbb{F}_p^*\rightarrow\mathbb{C}$, we extend $\chi$ to $\mathbb{F}_p$ by defining $\chi(0)=0$, and let $\ord(\chi)$ denote the order of $\chi$. 
	
	Given any integer $3\le k\le p$ and any $x_0,d\in\mathbb{F}_p$ with $d\in\mathbb{F}_p^*$, the sequence 
	$$x_0,x_0+d,x_0+2d,\cdots,x_0+(k-1)d$$
	is called a nontrivial arithmetic progression of length $k$ over $\mathbb{F}_p$. For any nonempty  subset $A\subseteq\mathbb{F}_p$, the characteristic function of $A$ is written as $1_A$, i.e., 
	$$1_A(x)=\begin{cases}
		1 & \mbox{if}\ x\in A,\\
		0 & \mbox{otherwise}.
	\end{cases}$$
	We use $\#A$ to denote the cardinality of $A$. Also, for any vector space $W$ over $\mathbb{F}_p$, let $\dim(W)$ be the dimension of $W$. 
	
	Let $f$ and $g$ be functions defined on the set of all primes. We adopt the standard notations in analytic number theory: $f(p)\ll g(p)$ (or $f=O(g)$) means that there exist a positive real number $C$ and a real number $x_0$ such that 
	$$|f(p)|\le C\cdot g(p)$$
	for every prime $p\ge x_0$. The symbol $f(p)\asymp g(p)$ means that $f(p)\ll g(p)$ and $g(p)\ll f(p)$. If $g(p)$ is non-vanishing for any prime $p\ge x_0$, then the symbol $f(p)=o(g(p))$ means that 
	$$\lim_{p\to+\infty}\frac{f(p)}{g(p)}=0.$$
	
	\subsection{Background and motivation}
	Verifying the existence of arithmetic progressions of length $k$ in a given set is a classical problem in additive combinatorics. For example, the well-known Green-Tao theorem \cite{GT} states that the set of primes contains arbitrary long nontrivial arithmetic progressions. Concerning $3$-term arithmetic progressions, the famous theorem of Roth \cite{Roth} states that if a subset $A\subseteq\{1,2,\cdots, N\}$ has no nontrivial $3$-term arithmetic progression, then 
	$$\#A=O(N/\log\log N).$$
	 By applying the polynomial method, Croot, Lev, and Pach \cite{CLP} studied the Roth-type problems in finite abelian groups. In fact, they showed that any subset $A\subseteq \left(\mathbb{Z}/4\mathbb{Z}\right)^n$ free of nontrivial $3$-term arithmetic progressions has size 
	$$\#A\le 4^{\gamma n},$$
	where $\gamma\approx0.926$. Later by extending the method of Croot, Lev, and Pach \cite{CLP}, Ellenberg and Gĳswĳt \cite{EG} proved that if $A\subseteq\mathbb{F}_3^n$ such that $A$ contains no nontrivial $3$-term arithmetic progressions, then $\#A=o(2.756^n)$. 
	
	On the other hand, the celebrated Szemer\'edi theorem \cite{Sze} shows that any dense subset of the positive integers contains arbitrarily long arithmetic progressions. Specifically, for a fixed integer $k\ge3$ and a fixed real number $0<\delta\le 1$, if $N$ is sufficiently large, then any subset $A\subseteq \mathbb{Z}/N\mathbb{Z}$ with density at least $\delta$ contains an arithmetic progression of length $k$. Later Conlon, Fox and Zhao \cite{CFZ} established the useful relative Szemer\'edi theorem. 
	
	Arithmetic progressions in finite fields have also been investigated extensively. For example, by applying the extension of Szemer\'edi’s theorem for pseudorandom weights due to Green and Tao \cite{GT}, Chang \cite{Chang} studied the arithmetic progressions in a multiplicative subgroup of $\mathbb{F}_p^*$. More precisely, given an integer $k\ge3$ and a real number $0<\delta\le 1$, if $p$ is sufficiently large and $H\le \mathbb{F}_p^*$ is a subgroup with $\#H>p^{1-\frac{1}{k\cdot 2^{k+1}}}$, then Chang \cite[Theorem 1]{Chang} proved that any subset $A\subseteq H$ with $\#A/\#H>\delta$ contains a nontrivial $k$-term arithmetic progression. 
	
	In addition to multiplicative groups of $\mathbb{F}_p^*$, the set $\mathcal{P}_p$ of all primitive elements of $\mathbb{F}_p$ is also an important object of study in finite fields. Cohen and his collaborators \cite{Cohen15} studied arithmetic progressions with common difference $1$ in $\mathcal{P}_p$. For example, they \cite[Theorem 1.1]{Cohen15} showed that $\mathcal{P}_p$ contains a $3$-term arithmetic progression of the form $x,x+1,x+2$ whenever $p>169$. Moreover, given an integer $k\ge4$, Cohen and his collaborators \cite[Theorem 2]{Cohen15} also showed that there is an element $x\in\mathbb{F}_p$ such that $x,x+1,\cdots,x+(k-1)\in\mathcal{P}_p$, provided that $p$ is sufficiently large. 
	
	Inspired by the above results, it is natural to consider the existence of arithmetic progressions of length $k$ in dense subsets of $\mathcal{P}_p$.

	\subsection{Main results} Now we state our first theorem.
	
	\begin{theorem}\label{Thm. A}
		Let $k\ge3$ be an integer and $0<\delta\le 1$ be a real number. Then for every sufficiently large prime $p$, any subset $A\subseteq\mathcal{P}_p$ with $\#A/\#\mathcal{P}_p\ge \delta$ contains a non-trivial arithmetic progression of length $k$ over $\mathbb{F}_p$. 
	\end{theorem}
	
	Let 
	$$\Omega_p=\left\{g\in\mathbb{Z}: 1\le g\le p-1\ \text{and $g$ is a primitive root modulo $p$}\right\}$$
	be the lift of $\mathcal{P}_p$ to $\mathbb{Z}$. As a direct consequence of Theorem \ref{Thm. A}, we have the following result.
	
	\begin{corollary}\label{Corollary of Thm. A}
		Let $k\ge3$ be an integer and $0<\delta\le 1$ be a real number. Then for every sufficiently large prime $p$, each subset $B\subseteq\Omega_p$ with $\#B/\#\Omega_p\ge\delta$ contains a nontrivial $k$-term arithmetic progression of length $k$ over $\mathbb{Z}$. 
	\end{corollary}

		Roughly speaking, Theorem \ref{Thm. A} illustrates that every
		subset $A\subseteq\mathcal P_p$ with 
		$$\#A\asymp (\#\mathcal{P}_p)$$ 
		contains a nontrivial arithmetic progression of length $k$ for all sufficiently large $p$. The next theorem will show that Theorem \ref{Thm. A} is optimal at the level of the power of $\#\mathcal{P}_p=\varphi(p-1)$:  the condition $\#A\asymp (\#\mathcal{P}_p)$ cannot be replaced by $\#A\asymp (\#\mathcal{P}_p)^{1-\varepsilon}$ for any fixed real number $\varepsilon\in(0,1)$.

	\begin{theorem}\label{Thm. B}
		Fix an integer $k\ge3$ and a real number $\varepsilon\in(0,1)$. For every sufficiently large prime $p$, there is a subset $A\subseteq\mathcal{P}_p$ with $\#A=\lfloor (\#\mathcal{P}_p)^{1-\varepsilon}\rfloor$ such that $A$ contains no nontrivial arithmetic progression of length $k$.
	\end{theorem}

	\subsection{Outline of the paper} In Section 2, we will briefly introduce the relative Szemer\'edi theorem established by Conlon, Fox and Zhao. The proofs of our main results will be given in Section 3--5 respectively. 
	
	\section{The relative Szemer\'edi theorem due to Conlon, Fox and Zhao}
	\setcounter{lemma}{0}
	\setcounter{theorem}{0}
	\setcounter{equation}{0}
	\setcounter{conjecture}{0}
	\setcounter{remark}{0}
	\setcounter{corollary}{0}
	
	Let $N$ be a sufficiently large integer. For any function $f(x_1,x_2,\cdots, x_m): \left(\mathbb{Z}/N\mathbb{Z}\right)^m\rightarrow\mathbb{R}$ and any nonempty subsets $A_1,\cdots,A_m\subseteq \mathbb{Z}/N\mathbb{Z}$, let 
	$$\mathbb{E}\left[f(x_1,\cdots,x_m): x_1\in A_1,\cdots, x_m\in A_m\right]$$
	be the expectation of $f(x_1,\cdots,x_m)$ when each $x_i$ is chosen uniformly and independently at random from $A_i$. 
	
	For any positive integer $k$, the symbol $[1,k]$ denotes the set $\{1,2,\cdots, k\}$. Conlon, Fox and Zhao  \cite[Definition 2.2]{CFZ} introduce the following $k$-linear forms condition.
	
	\begin{definition}[$k$-linear forms condition]\label{Def. linear forms condition}
		Let $k$ be a positive integer and let $N$ be a sufficiently large integer. A nonnegative function $v_N: \mathbb{Z}/N\mathbb{Z}\rightarrow\mathbb{R}_{\ge 0}$ is said to obey the $k$-linear forms condition if one has 
		$$\mathbb{E}\left[\prod_{j=1}^{k}\prod_{\substack{{\bm \omega}=(\omega_i)\\ i\in[1,k]\setminus\{j\} \\ \omega_i\in \{0,1\}}}v_N\left(\sum_{i\in[1,k]\setminus\{j\}}(i-j)x_i^{(\omega_i)}\right)^{n_{j,{\bm \omega}}}: x_1^{(0)},x_1^{(1)},\cdots, x_k^{(0)},x_k^{(1)}\in\mathbb{Z}/N\mathbb{Z}\right]=1+o(1),$$
		for any choices of exponents $n_{j,{\bm \omega}}\in\{0,1\}$. 
	\end{definition}
	
	Conlon, Fox and Zhao \cite[Theorem 2.4]{CFZ} established the following result. 
	
	\begin{theorem}[Conlon, Fox and Zhao]\label{Thm. CFZ}
		For every integer $k\ge3$ and real number $\delta>0$, there exists $c>0$ such that if $v_N: \mathbb{Z}/N\mathbb{Z}\rightarrow\mathbb{R}_{\ge 0}$ satisfies the $k$-linear forms condition, $N$ is sufficiently large, and $f: \mathbb{Z}/N\mathbb{Z}\rightarrow\mathbb{R}_{\ge 0}$ satisfies $0\le f(x)\le v_N(x)$ for all $x\in \mathbb{Z}/N\mathbb{Z}$ and $\mathbb{E}[f(x): x\in \mathbb{Z}/N\mathbb{Z}]\ge \delta$, then 
		$$\mathbb{E}\left[f(x)f(x+d)\cdots f(x+(k-1)d): x, d\in\mathbb{Z}/N\mathbb{Z}\right]\ge c.$$
	\end{theorem}

	\section{Proofs of Theorem \ref{Thm. A}}
	\setcounter{lemma}{0}
	\setcounter{theorem}{0}
	\setcounter{equation}{0}
	\setcounter{conjecture}{0}
	\setcounter{remark}{0}
	\setcounter{corollary}{0}
	
	Recall that $\mathcal{P}_p$ is the set of all primitive elements of $\mathbb{F}_p$. We begin with the following known result on the characteristic function of $\mathcal{P}_p$ (cf. \cite{Cohen15,Cohen21}).
	
	\begin{lemma}\label{Lem. characteristic function of P}
		Let $p$ be a prime. Then 
		$$1_{\mathcal{P}_p}(x)=\theta_{p-1}\sum_{d\mid p-1}\frac{\mu(d)}{\varphi(d)}\sum_{\substack{\chi\in\widehat{\mathbb{F}_p^*} \\ \ord(\chi)=d}} \chi(x)=\theta_{p-1}\sum_{\chi\in\widehat{\mathbb{F}_p^*}}c_{\chi}\cdot \chi(x)
		=\begin{cases}
			1  & \mbox{if}\ x\in\mathcal{P}_p,\\
			0 & \mbox{otherwise},
		\end{cases}$$
		where $\mu$ is the M\"obius function, $\varphi$ is the Euler totient function, $\theta_{p-1}=\frac{\varphi(p-1)}{(p-1)}$, and $c_{\chi}=\frac{\mu(\ord(\chi))}{\varphi(\ord(\chi))}$ for any $\chi\in\widehat{\mathbb{F}_p^*}$. 
	\end{lemma}
	
	We now define the function 
	\begin{equation}\label{Eq. definition of vp}
		v_p=\frac{1}{\theta_{p-1}}\cdot 1_{\mathcal{P}_p}: \mathbb{F}_p\rightarrow \mathbb{R}_{\ge 0}.
	\end{equation}
	
	Our next task is to show that $v_p$ satisfies the $k$-linear forms condition in Definition \ref{Def. linear forms condition}. To accomplish this task, we first introduce the well-known Weil theorem (cf. \cite[Theorem 5.41]{LN}).
	
	\begin{lemma}\label{Lem. the Weil Bound}
		Let $\chi\in\widehat{\mathbb{F}_p^*}$ with $\ord(\chi)=d>1$, and let $f(t)\in\mathbb{F}_p[t]$ be a monic polynomial with $f(t)\neq g(t)^d$ for any $g(t)\in\mathbb{F}_p[t]$. Then, for any $a\in\mathbb{F}_p$ we have 
		$$\left|\sum_{x\in\mathbb{F}_p}\chi(af(x))\right|\le (r-1)\sqrt{p},$$
		where $r$ is the number of distinct roots of $f(t)$ in an algebraic closure $\overline{\mathbb{F}_p}$ of $\mathbb{F}_p$.
	\end{lemma}
	
	Alon \cite[Theorem 1.2]{A} obtained the following useful result, which is known as the Combinatorial Nullstellensatz.
	
	\begin{lemma}\label{Lem. Alon}
		Let $F$ be an arbitrary field, and let $f(x_1,\cdots,x_m)\in F[x_1,\cdots,x_m]$. Suppose $\deg(f)=t_1+t_2+\cdots+t_m$, where each $t_i$ is a nonnegative integer, and suppose that the coefficient of $x_1^{t_1}\cdots x_m^{t_m}$ in $f$ is nonzero. Then, if $S_1,\cdots,S_m$ are subsets of $F$ with $\#S_i>t_i$, there are $s_1\in S_1, s_2\in S_2,\cdots, s_m\in S_m$ so that 
		$$f(s_1,s_2,\cdots,s_m)\neq 0.$$
	\end{lemma}
	
	Using Lemmas \ref{Lem. the Weil Bound}--\ref{Lem. Alon}, we can obtain the following result, which will play an important role in verifying that $v_p$ satisfies the $k$-linear forms condition.
	
	\begin{lemma}\label{Lem. estimates for nontrival character sums}
		Let $m,s\ge2$ be integers, and let $L_i(x_1,x_2,\cdots,x_m)\in\mathbb{Z}[x_1,x_2,\cdots,x_m]$ with $\deg(L_i)=1$ for any $1\le i\le s$. Suppose  that $L_i$ and $L_j$ are not proportional over $\mathbb{Q}$ for any $1\le i\neq j\le s$. Then, for every sufficiently large prime $p$ we have 
		\begin{equation}\label{Eq. inequality in Lemma estimates for nontrival character sums}
			\left|\sum_{\x\in\mathbb{F}_p^m}\prod_{i=1}^{s}\chi_i(L_i(\x))\right|\le (s-1)\cdot p^{m-\frac{1}{2}}+\frac{s(s-1)}{2}\cdot p^{m-1},
		\end{equation}
		where $\chi_1,\chi_2,\cdots,\chi_s\in\widehat{\mathbb{F}_p^*}$ are not all trivial. 
	\end{lemma}
	
	\begin{proof}
		For each integer $1\le i\le s$, let 
		$$L_i(\x)=a_{i1}x_1+a_{i2}x_2+\cdots+a_{im}x_m+b_i=\a_i\cdot \x+b_i,$$
		where $\a_i=(a_{i1},a_{i2},\cdots,a_{im})\neq \0$, and let $\widetilde{\a_i}=(a_{i1},a_{i2},\cdots,a_{im},b_i)$. 
		
		For any prime $p$, we may view $L_1, L_2, \cdots,L_s$ as polynomials over $\mathbb{F}_p$. We first show that for sufficiently large $p$, when viewed as polynomials over $\mathbb{F}_p$, we still have $\deg(L_i)=1$ for any $1\le i\le s$ and $L_i$ and $L_j$ are not proportional over $\mathbb{F}_p$ for any $1\le i\neq j\le s$. In fact, given any $1\le i\neq j\le s$, since $L_i$ and $L_j$ are not proportional over $\mathbb{Q}$, the vectors $\widetilde{\a_i}$ and $\widetilde{\a_j}$ are linearly independent over $\mathbb{Q}$. From this, the rank of the matrix $\begin{bmatrix}
			\widetilde{\a_i}\\
			\widetilde{\a_j}
		\end{bmatrix}$ is equal to $2$, and hence this matrix
		has a nonzero $2\times 2$ minor $D_{ij}$. Applying this and noting that 
		$$d_i=\max\left\{|a_{ij}|: 1\le j\le m\right\}>0$$
		for each $1\le i\le s$, 
		 when 
		$$p>p_0=\max\left\{s+1, \prod_{1\le i<j\le s}|D_{ij}|\cdot \prod_{1\le i\le s}d_i\right\},$$
		it is clear that $\deg(L_i)=1$ for any $1\le i\le s$ and $L_i$ and $L_j$ are not proportional over $\mathbb{F}_p$ for any $1\le i\neq j\le s$.
		
		We next show that for any $p>p_0$, there exists a vector $\v\in\mathbb{F}_p^m$ such that $\v\cdot \a_i\neq 0$ for any $1\le i\le s$. In fact, since $p>p_0$, the vectors $\a_i\in\mathbb{F}_p^m\setminus\{\0\}$ for each $1\le i\le s$. Thus, 
		$$\left\{\v\in\mathbb{F}_p^m: \v\cdot \a_i=0\right\}$$
		is a subspace of $\mathbb{F}_p^m$ with dimension $m-1$ for any $1\le i\le s$. From this and noting that $p>p_0>s$, one can verify that 
		\begin{align*}
			\#\left\{\v\in\mathbb{F}_p^m: \v\cdot \a_i=0\ \text{for some $1\le i\le s$}\right\}
&=\#\bigcup_{i=1}^s\left\{\v\in\mathbb{F}_p^m: \v\cdot \a_i=0\right\}\\
&\le \sum_{i=1}^s\#\left\{\v\in\mathbb{F}_p^m: \v\cdot \a_i=0\right\}\\
&=sp^{m-1}\\
&<p^m.
		\end{align*}
		Hence there exists a vector $\v\in\mathbb{F}_p^m$ such that $\v\cdot \a_i\neq 0$ for any $1\le i\le s$ whenever $p>p_0$. 
		
		From now on, we assume $p>p_0$. Since $m\ge2$, there is a subspace $W$ of $\mathbb{F}_p^m$ with $\dim(W)=m-1$ such that 
		$$\mathbb{F}_p^m=W\bigoplus\mathbb{F}_p\v.$$
		This implies that every vector $\x\in\mathbb{F}_p^m$ can be uniquely written as 
		$$\x=\w+t\v \quad (\w\in W, t\in\mathbb{F}_p).$$
		Hence, for any $\x=\w+t\v\in\mathbb{F}_p^m$ with $\w\in W$ and $t\in\mathbb{F}_p$ and any integer $1\le i\le s$, letting 
		\begin{equation}\label{Eq. definition of ci}
			c_i=\v\cdot\a_i\neq 0,
		\end{equation}
		one can verify that 
		\begin{align}\label{Eq. expression of Li}
			L_i(\x)=L_i(\w+t\v)
&=\a_i\cdot (\w+t\v)+b_i\notag\\
&=L_i(\w)+c_it\notag\\
&=c_i(t-\alpha_i(\w)),
		\end{align}
where 
      \begin{equation}\label{Eq. definition of alpha i}
	\alpha_i(\w)=-\frac{L_i(\w)}{c_i}.
       \end{equation}

Define 
$$W_{\bad}=\left\{\w\in W: \alpha_i(\w)=\alpha_j(\w)\ \text{for some $1\le i\neq j\le s$}\right\}$$
and let $W_{\good}=W\setminus W_{\bad}$. We next obtain an upper bound for $\# W_{\bad}$. For any $1\le i\neq j\le s$, let the polynomial
$$L_{ij}=c_jL_i-c_iL_j.$$
Since $p>p_0$, by the above discussions, the polynomials $L_i$ and $L_j$ are not proportional over $\mathbb{F}_p$. From this and recalling that  $c_i,c_j\neq 0$ are defined by (\ref{Eq. definition of ci}), we immediately obtain $L_{ij}\neq 0$. On the other hand, for any $\x=\w+t\v\in\mathbb{F}_p^m$ with $\w\in W$ and $t\in\mathbb{F}_p$, applying (\ref{Eq. expression of Li}) and (\ref{Eq. definition of alpha i}) one can verify that 
\begin{align}\label{Eq. Lij is independent of tv}
	L_{ij}(\x)=L_{ij}(\w+t\v)
&=c_jL_i(\w+t\v)-c_iL_j(\w+t\v)\notag\\
&=c_j\left(c_i(t-\alpha_i(\w))\right)-c_i\left(c_j(t-\alpha_j(\w))\right)\notag\\
&=-c_jc_i\alpha_i(\w)+c_ic_j\alpha_j(\w)\notag\\
&=c_jL_i(\w)-c_iL_j(\w)\notag\\
&=L_{ij}(\w).
\end{align}
We now claim that there is at least one $\w\in W$ such that $L_{ij}(\w)\neq 0$. Suppose, to the contrary, that $L_{ij}(\w)=0$ for any $\w\in W$. Then (\ref{Eq. Lij is independent of tv}) implies that $L_{ij}(\x)=0$ for any $\x\in\mathbb{F}_p^m$. Thus, applying Lemma \ref{Lem. Alon} with $F=\mathbb{F}_p$ and $S_1=\cdots=S_m=\mathbb{F}_p$, we obtain $L_{ij}=0$, which contradicts the fact that $L_{ij}\neq 0$. By (\ref{Eq. Lij is independent of tv}) and the above discussions, 
$$\left\{\w\in W: L_{ij}(\w)=0\right\}\subsetneqq W$$
is either $\emptyset$ or a proper affine hyperplane in $W$. From this and noting that 
$$\left\{\w\in W: \alpha_i(\w)=\alpha_j(\w)\right\}=\left\{\w\in W: L_{ij}(\w)=0\right\},$$
we obtain 
$$\#\left\{\w\in W: \alpha_i(\w)=\alpha_j(\w)\right\}\le p^{\dim(W)-1}=p^{m-2}.$$
Hence 
\begin{align}\label{Eq. upper bound for W bad}
	\#W_{\bad}
&=\#\left\{\w\in W: \alpha_i(\w)=\alpha_j(\w)\ \text{for some $1\le i\neq j\le s$}\right\}\notag\\
&=\#\bigcup_{1\le i<j\le s}\left\{\w\in W: \alpha_i(\w)=\alpha_j(\w)\right\}\notag\\
&\le \sum_{1\le i<j\le s}\#\left\{\w\in W: \alpha_i(\w)=\alpha_j(\w)\right\}\notag\\
&\le \binom{s}{2}p^{m-2}.
\end{align}

Finally, we prove the inequality (\ref{Eq. inequality in Lemma estimates for nontrival character sums}). As $\chi_1,\cdots,\chi_s$ are not all trivial, we have 
$$l=\lcm\{\ord(\chi_i): 1\le i\le s\}>1.$$
Fix a character $\psi\in\widehat{\mathbb{F}_p^*}$ with $\ord(\psi)=l$. Then there are integers $1\le e_1,\cdots, e_s\le l$ such that $\chi_i=\psi^{e_i}$ for any $1\le i\le s$. Since $\chi_1,\cdots,\chi_s$ are not all trivial, at least one of $e_1,\cdots,e_s$ is less than $l$. By this and (\ref{Eq. expression of Li}), one can verify that 
\begin{align}\label{Eq. sum is equal to good plus bad}
	\sum_{\x\in\mathbb{F}_p^m}\prod_{i=1}^{s}\chi_i(L_i(\x))
&=\sum_{\w\in W}\sum_{t\in\mathbb{F}_p}\prod_{i=1}^{s}\chi_i(L_i(\w+t\v))\notag\\
&=\sum_{\w\in W}\sum_{t\in\mathbb{F}_p}\prod_{i=1}^{s}\chi_i(c_i(t-\alpha_i(\w)))\notag\\
&=\sum_{\w\in W}\sum_{t\in\mathbb{F}_p}\prod_{i=1}^{s}\psi^{e_i}(c_i(t-\alpha_i(\w)))\notag\\
&=\sum_{\w\in W_{\good}}\sum_{t\in\mathbb{F}_p}\prod_{i=1}^{s}\psi^{e_i}(c_i(t-\alpha_i(\w)))+\sum_{\w\in W_{\bad}}\sum_{t\in\mathbb{F}_p}\prod_{i=1}^{s}\psi^{e_i}(c_i(t-\alpha_i(\w))).
\end{align}
	
	For any $\w\in W$, let 
	$$f_{\w}(x)=\prod_{1\le i\le s}\left(x-\alpha_i(\w)\right)^{e_i}\in\mathbb{F}_p[x].$$
	If $\w\in W_{\good}$, then $\alpha_1(\w),\alpha_2(\w),\cdots, \alpha_s(\w)$ are pairwise distinct. Noting that $1\le e_1,\cdots,e_s\le l$ and $\min\{e_i: 1\le i\le s\}<l$, we have 
		$$f_{\w}(x)\neq g(x)^l$$
		for any $g(x)\in\mathbb{F}_p[x]$.  Applying Lemma \ref{Lem. the Weil Bound}, for each $\w\in W_{\good}$ we obtain 
$$\left|\sum_{t\in\mathbb{F}_p}\prod_{i=1}^{s}\psi^{e_i}(c_i(t-\alpha_i(\w)))\right|=\left|\sum_{t\in\mathbb{F}_p}\psi\left(cf_{\w}(t)\right)\right|\le (s-1)\sqrt{p},$$
where 
$$c=\prod_{1\le i\le s}c_i^{e_i}\in\mathbb{F}_p^*.$$
Hence, 
\begin{align}\label{Eq. estimate for good w}
	     \left|\sum_{\w\in W_{\good}}\sum_{t\in\mathbb{F}_p}\prod_{i=1}^{s}\psi^{e_i}(c_i(t-\alpha_i(\w)))\right|
&\le  \sum_{\w\in W_{\good}}\left|\sum_{t\in\mathbb{F}_p}\prod_{i=1}^{s}\psi^{e_i}(c_i(t-\alpha_i(\w)))\right|\notag\\
&\le \#W_{\good}\cdot (s-1)\cdot \sqrt{p}\notag\\
&\le \#W\cdot (s-1)\cdot \sqrt{p}\notag\\
&=(s-1)\cdot p^{m-\frac{1}{2}}.
\end{align}		
On the other hand, by (\ref{Eq. upper bound for W bad}) one can verify that 
\begin{align}\label{Eq. estimate for bad w}
	\left|\sum_{\w\in W_{\bad}}\sum_{t\in\mathbb{F}_p}\prod_{i=1}^{s}\psi^{e_i}(c_i(t-\alpha_i(\w)))\right|
&\le  \sum_{\w\in W_{\bad}}\left|\sum_{t\in\mathbb{F}_p}\prod_{i=1}^{s}\psi^{e_i}(c_i(t-\alpha_i(\w)))\right|\notag\\
&\le \sum_{\w\in W_{\bad}}p\notag\\
&=\#W_{\bad}\cdot p \notag\\
&\le \binom{s}{2}\cdot p^{m-1}.
\end{align}
		
Combining (\ref{Eq. estimate for good w}) and (\ref{Eq. estimate for bad w}) with (\ref{Eq. sum is equal to good plus bad}), we obtain 
\begin{align*}
	 &\left|\sum_{\x\in\mathbb{F}_p^m}\prod_{i=1}^{s}\chi_i(L_i(\x))\right|\\
\le &\left|\sum_{\w\in W_{\good}}\sum_{t\in\mathbb{F}_p}\prod_{i=1}^{s}\psi^{e_i}(c_i(t-\alpha_i(\w)))\right|+\left|\sum_{\w\in W_{\bad}}\sum_{t\in\mathbb{F}_p}\prod_{i=1}^{s}\psi^{e_i}(c_i(t-\alpha_i(\w)))\right|\\
\le & (s-1)\cdot p^{m-\frac{1}{2}}+\frac{s(s-1)}{2}\cdot p^{m-1}.
\end{align*}
		
In view of the above, we have completed the proof. 
	\end{proof}
	
	Let 
	$$W_{p-1}=\#\left\{d\ge 1: d\ \text{is square-free and}\ d\mid p-1\right\}$$
	be the number of positive square-free divisors of $p-1$. We also need the following known result. 
	\begin{lemma}\label{Lem. W p-1 is very small}
		For any real number $\varepsilon>0$, if $p$ is sufficiently large, then $W_{p-1}\ll_{\varepsilon}p^{\varepsilon}$. 
	\end{lemma}
	
	\begin{proof}
     Let 
$$\omega_{n}=\#\left\{q: q\mid n\ \text{and}\ q\ \text{is a prime}\right\}$$
     be the number of all distinct prime divisors of $n$. Robin \cite[Theorem 11]{Robin} proved that for any integer $n\ge 3$, we have 
     	$$\omega_{n}<1.3841\frac{\log (n)}{\log\log (n)}.$$
     	Thus, applying this with $n=p-1$, we obtain
     	$$W_{p-1}=2^{\omega_{p-1}}=e^{\omega_{p-1}\cdot \log 2}\le (p-1)^{\frac{1.3841\cdot\log 2}{\log\log(p-1)}}\ll_{\varepsilon}p^{\varepsilon}$$
     	for any sufficiently large prime $p$. This completes the proof. 
	\end{proof}

	Now we are in a position to verify that the function $v_p$, defined by (\ref{Eq. definition of vp}), satisfies the $k$-linear forms condition. 
	
	\begin{lemma}\label{Lem. vp satisfies the k linear forms condition}
			Let $m,s$ be nonnegative integers with $m\ge2$, and let $L_i(x_1,x_2,\cdots,x_m)\in\mathbb{Z}[x_1,x_2,\cdots,x_m]$ with $\deg(L_i)=1$ for any $1\le i\le s$. Suppose  that $L_i$ and $L_j$ are not proportional over $\mathbb{Q}$ for any $1\le i\neq j\le s$. Then, for every sufficiently large prime $p$ we have 
			$$\mathbb{E}\left[\prod_{i=1}^{s}v_p(L_i(\x)): \x\in\mathbb{F}_p^m\right]=\frac{1}{p^m}\sum_{\x\in\mathbb{F}_p^m}\prod_{i=1}^{s}v_p(L_i(\x))=1+o(1).$$
		   In particular, for any integer $k\ge3$, the function $v_p$ satisfies the $k$-linear forms condition. 
	\end{lemma}
	
	\begin{proof}
		When $s=0$, i.e., no factor is selected, the expectation is exactly $1$. For $s=1$, as $p$ is sufficiently large, the reduction of $L_1$ modulo $p$ has degree $1$. Hence, for any $y\in\mathbb{F}_p$ we have 
		$$\#\left\{\x\in\mathbb{F}_p^m: L_1(\x)=y\right\}=p^{m-1}.$$
		Applying this and noting that
		$$\sum_{y\in\mathbb{F}_p}\chi(y)=
		\begin{cases}
			p-1 & \mbox{if}\ \chi=\chi_0,\\
			0    & \mbox{otherwise},
		\end{cases}$$
		we obtain 
		\begin{align*}
			\mathbb{E}\left[v_p(L_1(\x)): \x\in\mathbb{F}_p^m\right]
&=\frac{1}{p^m}\sum_{\x\in\mathbb{F}_p^m}\sum_{\chi\in\widehat{\mathbb{F}_p^*}}c_{\chi}\cdot \chi(L_1(\x))\\
&=\frac{1}{p^m}\sum_{\chi\in\widehat{\mathbb{F}_p^*}}\sum_{\x\in\mathbb{F}_p^m}c_{\chi}\cdot \chi(L_1(\x))\\
&=\frac{1}{p^m}\sum_{\chi\in\widehat{\mathbb{F}_p^*}}\sum_{y\in\mathbb{F}_p}c_{\chi}\cdot \chi(y)\#\left\{\x\in\mathbb{F}_p^m: L_1(\x)=y\right\}\\
&=\frac{1}{p}\sum_{\chi\in\widehat{\mathbb{F}_p^*}}c_{\chi}\sum_{y\in\mathbb{F}_p}\chi(y)\\
&=1-\frac{1}{p}\\
&=1+o(1).
		\end{align*}
		
		Suppose now $s\ge2$. By Lemma \ref{Lem. characteristic function of P} and (\ref{Eq. definition of vp}), one can verify that 
		\begin{align}\label{Eq. sum is trivial plus nontrivial}
			\sum_{\x\in\mathbb{F}_p^m}\prod_{i=1}^{s}v_p(L_i(\x))
&=\sum_{\x\in\mathbb{F}_p^m}\prod_{i=1}^{s}\sum_{\chi\in\widehat{\mathbb{F}_p^*}}c_{\chi}\cdot \chi(L_i(\x))\notag\\
&=\sum_{\x\in\mathbb{F}_p^m}\sum_{\chi_1,\cdots,\chi_s\in\widehat{\mathbb{F}_p^*}}\prod_{i=1}^{s}c_{\chi_i}\cdot \chi_i(L_i(\x))\notag\\
&=\sum_{\chi_1,\cdots,\chi_s\in\widehat{\mathbb{F}_p^*}}\prod_{i=1}^{s}c_{\chi_i} \sum_{\x\in\mathbb{F}_p^m}\prod_{i=1}^{s}\chi_i(L_i(\x))\notag\\
&=S_1+S_2,
		\end{align}
	where 
	$$S_1=\sum_{\x\in\mathbb{F}_p^m}\prod_{i=1}^{s}\chi_0(L_i(\x)),$$
	and
	$$S_2=\sum_{\substack{\chi_1,\cdots,\chi_s\in\widehat{\mathbb{F}_p^*}\\ \chi_1,\cdots,\chi_s\ \text{are not all trivial}}}\prod_{i=1}^{s}c_{\chi_i} \sum_{\x\in\mathbb{F}_p^m}\prod_{i=1}^{s}\chi_i(L_i(\x)).$$

		For $S_1$, we clearly have 
		\begin{equation}\label{Eq. upper bound for trivial}
			\sum_{\x\in\mathbb{F}_p^m}\prod_{i=1}^{s}\chi_0(L_i(\x))\le p^m.
		\end{equation}
		On the other hand, since $p$ is sufficiently large, the reduction of $L_i$ modulo $p$ has degree $1$. This implies 
		$$\#\left\{\x\in\mathbb{F}_p^m: L_i(\x)=0\right\}=p^{m-1}.$$
		From this, one can verify that 
		\begin{align}\label{Eq. lower bound for trivial}
			\sum_{\x\in\mathbb{F}_p^m}\prod_{i=1}^{s}\chi_0(L_i(\x))
&=\#\left\{\x\in\mathbb{F}_p^m: L_i(\x)\neq 0\ \text{for any $1\le i\le s$}\right\}\notag\\
&=p^m-\#\left\{\x\in\mathbb{F}_p^m: L_i(\x)=0\ \text{for some $1\le i\le s$}\right\}\notag\\
&=p^m-\#\bigcup_{1\le i\le s}\left\{\x\in\mathbb{F}_p^m: L_i(\x)=0\right\}\notag\\
&\ge p^m-\sum_{1\le i\le s}\#\left\{\x\in\mathbb{F}_p^m: L_i(\x)=0\right\}\notag\\
&\ge p^m-s\cdot p^{m-1}.
		\end{align}
		
		We next consider $S_2$. For any positive divisor $d$ of $p-1$, it is known that
		$$\#\left\{\chi\in\widehat{\mathbb{F}_p^*}: \ord(\chi)=d\right\}=\varphi(d).$$
		Hence, by Lemma \ref{Lem. characteristic function of P} we obtain 
		\begin{align}\label{Eq. sum of c chi}
			\sum_{\chi\in\widehat{\mathbb{F}_p^*}}|c_{\chi}|
&=\sum_{\chi\in\widehat{\mathbb{F}_p^*}}\frac{|\mu(\ord(\chi))|}{\varphi(\ord(\chi))}\notag\\
&=\sum_{d\mid p-1}\frac{|\mu(d)|}{\varphi(d)}\cdot \#\left\{\chi\in\widehat{\mathbb{F}_p^*}: \ord(\chi)=d\right\}\notag\\
&=\sum_{d\mid p-1}|\mu(d)|\notag\\
&=W_{p-1}.
		\end{align}
		Assembling (\ref{Eq. sum of c chi}) and Lemma \ref{Lem. estimates for nontrival character sums} gives 
		\begin{align}\label{Eq. upper bound for the nontrivial}
	 |S_2|
&\le \sum_{\substack{\chi_1,\cdots,\chi_s\in\widehat{\mathbb{F}_p^*}\\ \chi_1,\cdots,\chi_s\ \text{are not all trivial}}}\prod_{i=1}^{s}|c_{\chi_i}| \left|\sum_{\x\in\mathbb{F}_p^m}\prod_{i=1}^{s}\chi_i(L_i(\x))\right|\notag\\
&\le \left((s-1)\cdot p^{m-\frac{1}{2}}+\frac{s(s-1)}{2}\cdot p^{m-1}\right)\cdot \sum_{\chi_1,\cdots,\chi_s\in\widehat{\mathbb{F}_p^*}}\prod_{i=1}^{s}|c_{\chi_i}|\notag\\
&=\left((s-1)\cdot p^{m-\frac{1}{2}}+\frac{s(s-1)}{2}\cdot p^{m-1}\right)\cdot \prod_{i=1}^{s}\sum_{\chi_i\in\widehat{\mathbb{F}_p^*}}|c_{\chi_i}|\notag\\
&=\left((s-1)\cdot p^{m-\frac{1}{2}}+\frac{s(s-1)}{2}\cdot p^{m-1}\right)\cdot W_{p-1}^s.
		\end{align}
	Combining (\ref{Eq. upper bound for trivial}), (\ref{Eq. lower bound for trivial}) and (\ref{Eq. upper bound for the nontrivial}) with (\ref{Eq. sum is trivial plus nontrivial}), we obtain 
	$$ 1-\left(\frac{s-1}{\sqrt{p}}+\frac{s(s+1)}{2p}\right)\cdot W_{p-1}^s
	\le \frac{1}{p^m}\sum_{\x\in\mathbb{F}_p^m}\prod_{i=1}^{s}v_p(L_i(\x))\le 1+\left(\frac{s-1}{\sqrt{p}}+\frac{s(s-1)}{2p}\right)\cdot W_{p-1}^s.$$
	Applying Lemma \ref{Lem. W p-1 is very small} to the above inequality, we obtain 
	$$\mathbb{E}\left[\prod_{i=1}^{s}v_p(L_i(\x)): \x\in\mathbb{F}_p^m\right]=\frac{1}{p^m}\sum_{\x\in\mathbb{F}_p^m}\prod_{i=1}^{s}v_p(L_i(\x))=1+o(1).$$
		
	Finally, observe that each polynomial appearing in Definition \ref{Def. linear forms condition} has the form 
	$$L_{j,{\bm{\omega}}}=\sum_{i\in[1,k]\setminus\{j\}}(i-j)x_i^{(\omega_i)}.$$
	
	When $j\neq j'$, the polynomial $L_{j,{\bm{\omega}}}$ omits the variables $x_j^{(0)}$ and $x_j^{(1)}$, while $L_{j',{\bm{\omega}}'}$ contains exactly one of them. This implies that $L_{j,{\bm{\omega}}}$ and $L_{j',{\bm{\omega}'}}$ are not proportional. If $j=j'$ but $\bm{\omega}\neq \bm{\omega}'$, then for some $i\neq j$, one polynomial contains $x_i^{(0)}$ but not $x_i^{(1)}$, and the other precisely contains $x_i^{(1)}$ but not $x_i^{(0)}$. Hence, these two polynomials cannot be proportional. Applying the above discussions, we immediately see that $v_p$ satisfies the $k$-linear forms condition for any integer $k\ge3$.
		
	In view of the above, we have completed the proof. 
	\end{proof}
	
	To complete the proof of our first theorem, we need the following result due to Rosser and Schoenfeld \cite[Theorem 15]{RS}.
	
	\begin{lemma}\label{Lem. on low bound for theta}
		For any positive integer $n\ge 3$, we have 
		$$\frac{n}{\varphi(n)}<e^{\gamma}\cdot \log\log n+\frac{2.50637}{\log\log n}\ll \log\log n,$$
		where $\gamma$ is the Euler constant. 
	\end{lemma}

	With the above preparations, we now give the proof of our theorem.
	
	{\noindent\bfseries Proof of Theorem \ref{Thm. A}}. Recall that $v_p(x)=\frac{1}{\theta_{p-1}}\cdot 1_{\mathcal{P}_p}(x)$. For any subset $A\subseteq \mathcal{P}_p$ with $\#A/\#\mathcal{P}_p\ge \delta$, define the function 
	$$f_A(x)=\frac{1}{\theta_{p-1}}\cdot 1_A(x).$$
	Clearly $0\le f_A(x)\le v_p(x)$ for all $x\in\mathbb{F}_p$. Since $\#A/\#\mathcal{P}_p\ge \delta$ and $\theta_{p-1}=\varphi(p-1)/(p-1)$, when $p$ is sufficiently large, we have 
	$$\mathbb{E}\left[f_A(x): x\in\mathbb{F}_p\right]=\frac{1}{p}\sum_{x\in\mathbb{F}_p}f_A(x)=\frac{\#A}{\varphi(p-1)}\cdot\frac{p-1}{p}=\frac{\#A}{\#\mathcal{P}_p}\cdot\frac{p-1}{p}\ge \delta/2.$$
	Since $v_p$ satisfies the $k$-linear forms condition, applying Theorem \ref{Thm. CFZ}, there is a real number $\lambda=\lambda(k,\delta/2)>0$ such that 
	$$\mathbb{E}\left[f_A(x)f_A(x+d)\cdots f_A(x+(k-1)d): x, d\in\mathbb{F}_p\right]\ge \lambda.$$
	Combining this with Lemma \ref{Lem. on low bound for theta}, when $p$ is sufficiently large, one can verify that 
	\begin{align*}
		\sum_{x,d\in\mathbb{F}_p}1_A(x)1_A(x+d)\cdots 1_A(x+(k-1)d)
&\ge \lambda\cdot p^2\cdot \theta_{p-1}^k\\
&=\lambda\cdot \frac{p^2\cdot \varphi(p-1)^k}{(p-1)^k}\\
&\gg \frac{p^2}{(\log\log (p-1))^k}\\
&>p.
	\end{align*}
	From this and noting that 
	$$\sum_{x\in\mathbb{F}_p}1_A(x)^k=\sum_{x\in\mathbb{F}_p}1_A(x)=\#A\le p,$$
	there exist $x,d\in\mathbb{F}_p$ with $d\in\mathbb{F}_p^*$ such that 
	$$x, x+d, \cdots, x+(k-1)d\in A,$$
	that is, $A$ contains a nontrivial arithmetic progression of length $k$. 
	
	In view of the above, we have completed the proof. \qed

	\section{Proofs of Corollary \ref{Corollary of Thm. A}}
	\setcounter{lemma}{0}
	\setcounter{theorem}{0}
	\setcounter{equation}{0}
	\setcounter{conjecture}{0}
	\setcounter{remark}{0}
	\setcounter{corollary}{0}
	
	Let $p$ be a prime. For any integer $x$, we use the symbol $\bar{x}$ to denote its reduction modulo $p$, i.e., $\bar{x}=x\mod{p}\in\mathbb{F}_p$. We begin with the following result.
	
	\begin{lemma}\label{Lem. lift mod p to Z}
		Let $k\ge3$ be an integer and let $p$ be an odd prime. Let $S\subseteq\mathbb{Z}$ with $\#S\ge2$ and 
		$$\diam(S)=\sup\left\{|x-y|: x,y\in S\ \text{and}\ x\neq y\right\}<\frac{p}{2}.$$
		Suppose that there are $b_0,b_1,\cdots,b_{k-1}\in S$ such that 
		$$\overline{b_0},\overline{b_1},\cdots,\overline{b_{k-1}}$$
		forms a nontrivial $k$-term arithmetic progression over $\mathbb{F}_p$. Then $b_0,b_1,\cdots,b_{k-1}$ is a nontrivial $k$-term arithmetic progression over $\mathbb{Z}$. 
	\end{lemma}
	
	\begin{proof}
		Since $\overline{b_0},\overline{b_1},\cdots,\overline{b_{k-1}}$ forms a nontrivial $k$-term arithmetic progression over $\mathbb{F}_p$, there is an integer $d$ with $d\not\equiv 0\pmod{p}$ such that 
		$$b_j\equiv b_0+jd\pmod{p}$$
		for any $0\le j\le k-1$. Thus, for any integer $j\in[0,k-3]$ we have 
		$$b_j-2b_{j+1}+b_{j+2}\equiv 0\pmod{p}.$$
		Noting that 
		$$\left|b_j-2b_{j+1}+b_{j+2}\right|\le \left|b_j-b_{j+1}\right|+\left|b_{j+2}-b_{j+1}\right|\le 2\cdot \diam(S)<p,$$
		the above congruence yields 
		$$b_j-2b_{j+1}+b_{j+2}=0$$
		for any $0\le j\le k-3$. This, together with $b_1-b_0\equiv d\not\equiv 0\pmod{p}$, implies that 
		$$b_0,b_1,\cdots,b_{k-1}$$ 
		forms a nontrivial $k$-term arithmetic progression over $\mathbb{Z}$. 
	\end{proof}
	
	Now we are in a position to prove our second result.
	
	{\noindent\bfseries Proof of Corollary \ref{Corollary of Thm. A}}. As $p$ is sufficiently large, we may assume $p>2$. Let 
	$$S_1=\left\{x\in\mathbb{Z}: 0<x<p/2\right\},\ \text{and}\ S_2=\left\{x\in\mathbb{Z}: p/2<x<p\right\}.$$
	Since $B=(B\cap S_1)\cup (B\cap S_2)$ and $(B\cap S_1)\cap (B\cap S_2)=\emptyset$, there is an $i\in\{1,2\}$ such that 
	$$\#(B\cap S_i)\ge \frac{\#B}{2}\ge\frac{\delta}{2}\cdot \#\Omega_p.$$
	Let $S=B\cap S_i$ and let $S_p$ be the reduction of $S$ modulo $p$, i.e., 
	$$S_p=\left\{\bar{x}: x\in S\right\}.$$
	Since $S\subseteq\{1,2,\cdots,p-1\}$, by the above discussion we have 
	$$\#S_p=\# S\ge\frac{\delta}{2}\cdot \#\Omega_p.$$
	Applying Theorem \ref{Thm. A}, there are $b_0,b_1,\cdots,b_{k-1}\in S$ such that 
	$$\overline{b_0},\overline{b_1},\cdots,\overline{b_{k-1}}$$
	forms a nontrivial $k$-term arithmetic progression over $\mathbb{F}_p$. From this and noting that $\diam(S)<p/2$, it follows from Lemma \ref{Lem. lift mod p to Z} that 
	$$b_0,b_1,\cdots,b_{k-1}\in S$$
	forms a nontrivial $k$-term arithmetic progression over $\mathbb{Z}$. 
	
	In view of the above, we have completed the proof. \qed

	\section{Proof of Theorem \ref{Thm. B}}
	\setcounter{lemma}{0}
	\setcounter{theorem}{0}
	\setcounter{equation}{0}
	\setcounter{conjecture}{0}
	\setcounter{remark}{0}
	\setcounter{corollary}{0}

	We begin with the well-known result due to Behrend \cite{Behrend}.
	
	\begin{lemma}[Behrend]\label{Lem. the Behrend theorem}
		There is an absolute constant $C>0$ such that for every sufficiently large integer $N$, there is a subset 
		$$B\subseteq\left\{1,2,\cdots,N\right\}\subseteq\mathbb{Z}^+$$
		with $\#B\ge N\cdot\exp\left(-C\sqrt{\log N}\right)$ such that $B$ contains no nontrivial $3$-term arithmetic progression over $\mathbb{Z}$. 
	\end{lemma}
	
	Now we prove our last theorem.
	
	{\noindent\bfseries Proof of Theorem \ref{Thm. B}}. Let $\varepsilon\in(0,1)$ be a fixed real number. Suppose that $p$ is sufficiently large. Applying Lemma \ref{Lem. the Behrend theorem}, there exists a subset 
	$$B\subseteq\left\{1,2,\cdots,\frac{p-1}{2}\right\}\subseteq\mathbb{Z}^+$$
	with 
	$$\#B\ge \frac{p-1}{2}\cdot\exp\left(-C\sqrt{\log ((p-1)/2)}\right)$$
	such that $B$ contains no nontrivial $3$-term arithmetic progression over $\mathbb{Z}$, where $C>0$ is an absolute constant appearing in Lemma \ref{Lem. the Behrend theorem}. By computations we clearly have 
	\begin{equation}\label{Eq. lower bound for B in the proof of Thm. B}
		\#B\gg p^{1-\frac{C_1}{\sqrt{\log p}}},
	\end{equation}
	where $C_1>0$ is independent of $p$. 
	
	Let $B_p=\{b\mod{p\mathbb{Z}}: b\in B\}$ be the reduction of $B$ modulo $p$. We claim that $B_p$ has no nontrivial $3$-term arithmetic progression over $\mathbb{F}_p$. In fact, suppose that there are $b_1,b_2,b_3\in B$ such that 
    $$b_1+b_3\equiv 2b_2\pmod{p}.$$
	Noting that $2\le b_1+b_3\le p-1$ and $2\le 2b_2\le p-1$, the above congruence yields 
	$$b_1+b_3=2b_2,$$
	that is, $b_1,b_2,b_3$ form a $3$-term arithmetic progression in $B$. Since $B$ contains no nontrivial $3$-term arithmetic progression over $\mathbb{Z}$, we have $b_1=b_2=b_3$ and hence $B_p$ has no nontrivial $3$-term arithmetic progression over $\mathbb{F}_p$. 
	
	For any $x\in\mathbb{F}_p$, let 
	$$B_p(x)=B_p+x=\left\{b_p+x: b_p\in B_p\right\}.$$
	We next compute the average of $\#(B_p(x)\cap\mathcal{P}_p)$, where $x$ ranges over $\mathbb{F}_p$. By Lemma \ref{Lem. characteristic function of P}, one can verify that 
	\begin{align*}
		\frac{1}{p}\sum_{x\in\mathbb{F}_p}\#(B_p(x)\cap\mathcal{P}_p)
&=\frac{1}{p}\sum_{x\in\mathbb{F}_p}\sum_{b_p\in B_p}1_{\mathcal{P}_p}(b_p+x)\\
&=\frac{1}{p}\sum_{b_p\in B_p}\sum_{x\in\mathbb{F}_p}1_{\mathcal{P}_p}(b_p+x)\\
&=\frac{1}{p}\sum_{b_p\in B_p}\sum_{x\in\mathbb{F}_p}1_{\mathcal{P}_p}(x)\\
&=\frac{1}{p}\sum_{b_p\in B_p}\#\mathcal{P}_p\\
&=\frac{\varphi(p-1)}{p}\#B_p\\
&=\frac{\varphi(p-1)}{p}\#B.
	\end{align*}
	Thus, there is an $x_0\in\mathbb{F}_p$ such that 
	$$\#A_0\ge \frac{\varphi(p-1)}{p}\#B,$$
	where $A_0=B_p(x_0)\cap\mathcal{P}_p$. From this, applying (\ref{Eq. lower bound for B in the proof of Thm. B}) and Lemma \ref{Lem. on low bound for theta}, for the fixed number $\varepsilon\in(0,1)$, one can verify that 
	\begin{align}\label{Eq. A0 is large}
		  \frac{\#A_0}{\varphi(p-1)^{1-\varepsilon}}
&\ge \frac{\varphi(p-1)^{\varepsilon}}{p}\#B\notag\\
&\gg \frac{\varphi(p-1)^{\varepsilon}}{p^{C_1/\sqrt{\log p}}}\notag\\
&=\frac{p^{\varepsilon}}{p^{C_1/\sqrt{\log p}}}\frac{\varphi(p-1)^{\varepsilon}}{p^{\varepsilon}}\notag\\
&\gg \frac{p^{\varepsilon}}{p^{C_1/\sqrt{\log p}}\cdot (\log\log p)^{\varepsilon}}.
	\end{align}
	Since 
	$$\lim_{p\rightarrow+\infty}\frac{p^{\varepsilon}}{p^{C_1/\sqrt{\log p}}\cdot (\log\log p)^{\varepsilon}}=+\infty,$$
	and $\varepsilon\in(0,1)$ is fixed, when $p$ is sufficiently large, the inequality (\ref{Eq. A0 is large}) implies that 
	$$\#A_0\ge \varphi(p-1)^{1-\varepsilon}.$$
	Choose a subset $A\subseteq A_0$ with $\#A=\lfloor\varphi(p-1)^{1-\varepsilon}\rfloor$. Since $A\subseteq A_0=(B_p+x_0)\cap\mathcal{P}_p$ and $B_p+x_0$ contains no nontrivial $3$-term arithmetic progression, we see that $A$ has no nontrivial $k$-term arithmetic progression for any integer $k\ge3$. 
	
	In view of the above, we have completed the proof. \qed 
	
   \Ack\quad This research was supported by the National Natural Science Foundation of China (Grant No. 12671009).

\end{document}